\documentclass[11pt,letterpaper,twoside]{article}
\usepackage[centertags]{amsmath}
\usepackage{graphicx,indentfirst,amsmath,amsfonts,amssymb,amsthm,newlfont}
\font\Goth=yinitas scaled \magstep0

\newcommand{\Gth}[1]{\lower2mm\hbox{\Goth #1}}

\usepackage[latin1]{inputenc}
\def\al{\alpha}

\def\de{\delta}

\def\be{\beta}

\def\l1{{\lambda}_1}

\newcommand{\f}{\frac}

\def\x1{{\xi }_{xx}}
\def\x2{{\xi }_{yy}}
\def\x3{{\xi }_{xy}}

\def\e1{{\eta }_{xx}}
\def\e2{{\eta }_{yy}}
\def\e3{{\eta }_{xy}}

\newcommand{\ds}{\displaystyle }
\newtheorem{definition}{Definition}
\newtheorem{theorem}{Theorem}

\newcommand{\beqn}{\begin{eqnarray*}}
\newcommand{\eeqn}{\end{eqnarray*}}
\newcommand{\beqnn}{\begin{eqnarray}}
\newcommand{\eeqnn}{\end{eqnarray}}
\newcommand{\p}{\partial}
\newcommand{\bb}{\begin{equation}}
\newcommand{\ee}{\end{equation}}
\newcommand{\ba}{\begin{array}}
\newcommand{\ea}{\end{array}}
\newcommand{\R}{\mathbb{R}}
\newcommand{\N}{\mathbb{N}}

\begin{document}
\pagenumbering{arabic}
\title{\huge \bf Self-adjoint sub-classes of third and fourth-order evolution equations}
\author{\rm \large Igor Leite Freire \\
\\
\it Centro de Matemática, Computação e Cognição\\ \it Universidade Federal do ABC - UFABC\\ \it 
Rua Catequese, $242$,
Jardim,
$09090-400$\\\it Santo André, SP - Brasil\\
\rm E-mail: igor.freire@ufabc.edu.br\\}
\date{\ }
\maketitle
\vspace{1cm}
\begin{abstract}
In this work a class of self-adjoint quasilinear third-order evolution equations is determined. Some conservation laws of them are established and a generalization on a self-adjoint class of fourth-order evolution equations is presented.
\end{abstract}
\vskip 1cm
\begin{center}
{2000 AMS Mathematics Classification numbers:\vspace{0.2cm}\\
76M60, 58J70, 35A30, 70G65\vspace{0.2cm} \\
Key words: Evolution equations, self-adjoint equation, conservation laws for evolution equations}
\end{center}
\pagenumbering{arabic}
\newpage

\section{Introduction}

In this paper we consider the problem on self-adjointness condition of equation
\bb\label{1.1.1}
u_{t}=r(u)u_{xxx}+p(u)u_{xx}+q(u)u_{x}^{2}+a(u)u_{x}+b(u),
\ee
where $r,\,p,\,q,\,a,\,b:\R\rightarrow\R$ are arbitrary smooth functions. 

Equation (\ref{1.1.1}) includes important evolution equations employed in mathematical physics and in mathematical biology, for instance, inviscid Burgers equation, Burgers equation, potential Burgers equation, Fisher equation, Korteweg--de Vries (KdV) equation, Gardner equation and so on, see \cite{cher, gun, igor1}. It can be used to describe shallow watter waves, collisionless-plasma magnetohydrodynamics waves and ion acoustic waves, among other physical or biological phenomena, see also \cite{ott, zabu}.

Similar work has been performed by Bruzón, Gandarias and Ibragimov \cite{ib1} regarding equation
\bb\label{1.1.3}
u_{t}+f(u)u_{xxxx}+g(u)u_{x}u_{xxx}+h(u)u_{xx}^{2}+d(u)u_{x}^{2}u_{xx}-p(u)u_{xx}-q(u)u_{x}^{2}=0.
\ee

However, in (\ref{1.1.3}) source terms and nonlinearities type $a(u)u_{x}$ and $r(u)u_{xxx}$ were not taken into account. So we shall complement the results previously obtained by them including these terms. 

Ibragimov \cite{ib2} has recently established a new conservation theorem for equations without Lagrangians. If (\ref{1.1.1}) is self-adjoint it is possible to construct conservation laws $D_{t}C^{0}+D_{x}C^{1}=0$ for it, where the components $C^{0}$ and $C^{1}$ depend on $t,x,u$ and its derivatives.

The purpose of this paper is to determine the self-adjoint equations type (\ref{1.1.1}) and, by using the recent result \cite{ib2}, establish some nontrivial conservation laws for some of these equations. The results on self-adjointness condition of equation (\ref{1.1.3}) obtained in \cite{ib1} is also generalized by including dispersive, convective and source terms.

The paper is organized as the follows: in the section \ref{review} we revisit some results regarding Lie point symmetries and conservation laws for differential equations. Section \ref{self} is devoted to find the self-adjoint equations type (\ref{1.1.1}). We comment some results presented in \cite{ib1} in section \ref{comment}.  

\section{Preliminaries}\label{review}

This section contains a brief discussion on the space of differential functions ${\cal A}$, Lie-Bäcklund operators, self-adjoint equations and conservation laws for differential equations. For more details, see \cite{ib0, ib1, ib2, ib3}. In the following the summation over repeated indices is understood.

Let $x=(x^{1},\cdots,x^{n})$  be $n$ independent variables and $u=(u^{1},\cdots,u^{m})$ be $m$ dependent variables with partial derivatives $u^{\al}_{i}=\f{\p u^{\al}}{\p x^{i}},\,u^{\al}_{ij}=\f{\p^{2}u^{\al}}{\p x^{i}\p x^{j}}$, etc. The total differentiation operators are given by
$$
D_{i}=\f{\p}{ \p x^{i}}+u^{\al}_{i}\f{\p}{\p u^{\al}}+u^{\al}_{ij}\f{\p}{\p u^{\al}_{j}}+\cdots,\,\,i=1,\cdots,n,\,\al=1,\cdots,m.
$$

Observe that $u^{\al}_{i}=D_{i}u^{\al},\,u^{\al}_{ij}=D_{i}D_{j}u^{\al}$, etc. The variables $u^{\al}$ with the sucessive derivatives $u^{\al}_{i_{1}\cdots i_{k}},\,k\in\N$, is known as the differential variables.

\begin{definition}
A locally analytic function of a finite number of the variables $x,\,u$ and $u$ derivatives is called a differential function. The highest order of derivatives appearing in the differential function is called the order of this function. The vector space of all differential functions of finite order is denoted by ${\cal A}$.
\end{definition}

{\bf Example}: Let us consider the differential function
\bb\label{2.1.1}
F=u_{t}-r(u)u_{xxx}-p(u)u_{xx}-q(u)u_{x}^{2}-a(u)u_{x}-b(u),
\ee
where $p,\,q,\,r,\,a,\,b:\R\rightarrow\R$ are arbitrary smooth functions. Supposing $r(u)\neq0$, the order of $F$ is three.

\begin{definition}
A Lie-Bäcklund operator is a formal sum
\bb\label{2.1.2}
X=\xi^{i}\f{\p}{\p x^{i}}+\eta^{\al}\f{\p}{\p u^{\al}}+\eta^{\al}_{i}\f{\p}{\p u^{\al}_{i}}+\eta^{\al}_{ij}\f{\p}{\p u^{\al}_{ij}}+\cdots,
\ee
where $\xi^{i},\,\eta^{\al}\in{\cal A},\, \eta^{\al}_{i}=D_{i}(\eta^{\al}-\xi^{j}u^{\al}_{j})+\xi^{j}u^{\al}_{ij},\,\eta^{\al}_{ij}=D_{i}D_{j}(\eta^{\al}-\xi^{k}u^{\al}_{k})+\xi^{k}u^{\al}_{kij}$, etc. 

The Lie-Bäcklund operator is often written as 
\bb\label{2.1.3}
X=\xi^{i}\f{\p}{\p x^{i}}+\eta^{\al}\f{\p}{\p u^{\al}}
\ee
understanding the prolonged form $(\ref{2.1.2})$. If $\xi^{i}=\xi^{i}(x,u)$ and $\eta=\eta(x,u)$ in $(\ref{2.1.3})$, then $X$ is a generator of Lie point symmetry group.
\end{definition}

{\bf Example:} The field 
\bb\label{4.1.1'}
X=t\f{\p}{\p t}-u\f{\p }{\p u}
\ee
is a Lie point symmetry generator of inviscid Burgers equation
\bb\label{4.1.1}
u_{t}=uu_{x}.
\ee

The set of all Lie-Bäcklund operators endowed with the commutator
$$
[X,Y]=(X(\zeta^{i})-Y(\xi^{i}))\f{\p}{\p x^{i}}+(X(\omega^{\al})-Y(\eta^{\al}))\f{\p}{\p u^{\al}}+\cdots,
$$
where $X$ is given by $(\ref{2.1.3})$ and 
$$Y=\zeta^{i}\f{\p}{\p x^{i}}+\omega^{\al}\f{\p}{\p u^{\al}},$$
is an infinite-dimensional Lie algebra.

\begin{definition}
The Euler-Lagrange operator $\f{\de}{\de u^{\al}}:{\cal A}\rightarrow{\cal A}$ is defined by the formal sum
\bb\label{2.1.4}
\f{\de}{\de u^{\al}}=\f{\p}{\p u^{\al}}+\sum_{j=1}^{\infty}(-1)^{j}D_{i_{1}}\cdots D_{i_{j}}\f{\p}{\p u^{\al}_{i_{1}\cdots i_{j}}}.
\ee
\end{definition}

\begin{definition}
Let $F_{\al}\in{\cal A}$. We define the adjoint system of differential functions $F_{\al}^{\ast}$ to $F_{\al}$ by the expression
$$
F^{\ast}_{\al}=\f{\de}{\de u^{\al}}(v^{\be}F_{\be}),
$$
where $v^{\be}$ is a new dependent variable. We say that $F_{\al}^{\ast}$ is a self-adjoint system of differential functions to $F_{\al}$ if there exists $\phi\in{\cal A}$ such that
$$
\left.F^{\ast}_{\al}\right|_{v=u}=\phi F_{\al}.
$$
\end{definition}

We observe that a system of differential equations can be viewed as $F_{\al}=0$, for some $F_{\al}\in{\cal A}$. 

\begin{definition}
An adjoint system of differential equations $F^{\ast}_{\al}=0$ to a system of differential equations $F_{\al}=0$ is given by
$$
F^{\ast}_{\al}=\f{\de}{\de u^{\al}}(v^{\be}F_{\be})=0,
$$
where $v^{\be}$ is a new dependent variables. We say that $F^{\ast}_{\al}=0$ is a self-adjoint equation to $F_{\al}=0$ if
\bb\label{ast}
\left.F^{\ast}_{\al}\right|_{v=u}=\phi F_{\al},
\ee
for some differential function $\phi\in{\cal A}$. So $\left.F^{\ast}_{\al}\right|_{v=u}=0$ if and only if $F_{\al}=0$.
\end{definition}

{\bf Example}: The adjoint differential function $F^{\ast}$ to (\ref{2.1.1}) is 
$$F^{\ast}=\f{\de}{\de u}(vF)=v[-p'u_{xx}-q'u_{x}^{2}-a'u_{x}-b']-D_{t}(v)-D_{x}[-v(2qu_{x}+a)]+D_{x}^{2}(-vp)-D_{x}^{3}(-vr).$$
Setting $v=u$ and after a tedious calculation we obtain
\bb\label{adj}
\ba{lcl}
F^{\ast}&=&-u_{t}-ub'(u)+a(u)u_{x}+[uq'(u)-2p'(u)-up''(u)+2q(u)]u_{x}^{2}+(3r''+ur''')u_{x}^{3}\\
\\
&&+[2uq(u)-2up'(u)+2uq(u)]u_{xx}+(6r'+3ur'')u_{x}u_{xx}+ru_{xxx}.
\ea\ee

If an equation possesses variational structure, it is well known that the Noether Theorem can be employed in order to establish conservation laws for the respective equation, {\it e.g.}, see \cite{yi1, yi2, igor2, naz}. 

However, the Noether Theorem cannot be applied to evolution equations in order to obtain conservation laws, since this class of equations does not possess variational structure. Fortunately there are some other alternative methods to establish conservation laws for equations without Lagrangians, see \cite{naz}. One of them is a recent result \cite{ib2}, due to Ibragimov.

Let
\bb\label{2.1.7}
X=\tau(t,x,u)\f{\p}{\p t}+\xi(t,x,u)\f{\p}{\p x}+\eta(t,x,u)\f{\p}{\p u}
\ee
be a Lie point symmetry generator of (\ref{1.1.1}) and
\bb\label{2.1.8}
{\cal L}=v F,
\ee
where $F$ is given by $(\ref{2.1.1})$. From the new conservation theorem \cite{ib2}, the conservation law for the system given by equation $(\ref{2.1.1})$ and by its adjoint equation $F^{\ast}=0$, where $F^{\ast}$ is given by $(\ref{adj})$,
is $Div(C)=D_{t}C^{0}D_{x}C^{1}=0$, where  
\bb\label{2.1.10}
\ba{lcl}
C^{0}&=&\ds{\tau {\cal L}+W\,\f{\p {\cal L}}{\p u_{t}}},\\
\\
C^{1}&=&\ds{\xi {\cal L}+W\left[\f{\p {\cal L}}{\p u_{x}}-D_{x}\f{\p {\cal L}}{\p u_{xx}}+D_{x}^{2}\f{\p {\cal L}}{\p u_{xxx}}\right]+D_{x}(W)\left[\f{\p {\cal L}}{\p u_{xx}}-D_{x}\f{\p {\cal L}}{\p u_{xxx}}\right]+D_{x}^{2}(W)\f{\p {\cal L}}{\p u_{xxx}}}
\ea
\ee
and $W=\eta-\tau u_{t}-\xi u_{x}$.

In particular, whenever $(\ref{2.1.1})$ is self-adjoint, substituting $v=u$ into $(\ref{2.1.10})$, $C=(C^{0},C^{1})$ provides a conserved vector for $(\ref{2.1.1})$.

\section{Self-adjoint equations type $(\ref{1.1.1})$}\label{self}

\subsection{The class of self-adjoint equations type (\ref{1.1.1})}

By applying the Euler-Lagrange operator (\ref{2.1.4}) to (\ref{2.1.8}), where $F$ is given by (\ref{2.1.1}) and equating to $0$, we obtain the adjoint equation to (\ref{1.1.1}), that is, $F^{\ast}$=0, where $F^{\ast}$ is given by (\ref{adj}).

Supposing that $F$ is self-adjoint, equation (\ref{ast}) holds, for some $\phi\in{\cal A}$. Comparing the coefficient of $u_{t}$, we obtain $\phi=-1$ and 
\bb\label{3.1.1}
\ba{lcl}
3r''+ur'''=0,\,\,\,\,
3ur''+6r'=0,\,\,\,\,
ur'''+3r''=0,\\ \\
-up'-p=-uq,\,\,\,\,
uq'-2p'-up''+q=0,\,\,
\,\,
ub'=-b.
\ea
\ee

Solving the system (\ref{3.1.1}), we obtain
\bb\label{3.1.1'}
r=a_{1}+\f{a_{2}}{u},
\ee
\bb\label{3.1.2}
q=\f{(up)'}{u}
\ee
and 
\bb\label{3.1.3}
b=\f{a_{3}}{u},
\ee
where $a_{1},\,a_{2}$ and $a_{3}$ are arbitrary constants. 

The following theorem is proved.

\begin{theorem}\label{teo1}
Equation $(\ref{1.1.1})$ is self-adjoint if and only if it has the form
\bb\label{3.2.1}
u_{t}=\left(a_{1}+\f{a_{2}}{u}\right)u_{xxx}+p(u)u_{xx}+\f{(up)'}{u}u_{x}^{2}+a(u)u_{x}+\f{a_{3}}{u},
\ee
where $a_{1},\,a_{2}$ and $a_{3}$ are constants.
\end{theorem}

\subsection{Conservation laws for equations type (\ref{3.2.1})}\label{examples}

Here we shall illustrate Theorem \ref{teo1} by using it and the results due to Ibragimov in order to establish some conservation laws for self-adjoint equations type (\ref{3.2.1}).

\subsubsection{Inviscid Burgers equation} 

The vector field (\ref{4.1.1'}) is a Lie point symmetry of the inviscid Burgers equation (\ref{4.1.1})(for more details, see \cite{igor1}). Since it is a self-adjoint equation, taking $v=u$ in (\ref{2.1.10}), the conservation law $D_{t}C^{0}+D_{x}C^{1}=0$ is obtained, where
$$
C^{0}=-u^{2}-tu^{2}\,u_{x},\,\,\,\,C^{1}=u^{3}+tu^{2}\, u_{t}.
$$

However, 
$$C^{0}=-u^{2}+D_{x}\left(-\f{tu^{3}}{3}\right),\,\,\,\, C^{1}=u^{3}+tD_{t}\left(\f{u^{3}}{3}\right)$$
and 
$$
\ba{lcl}
D_{t}C^{0}+D_{x}C^{1}&=&\ds{-D_{x}\left(\f{u^{3}}{3}\right)-tD_{t}D_{x}\left(\f{u^{3}}{3}\right)-D_{t}(u^{2})+D_{x}(u^{3})+tD_{x}D_{t}\left(\f{u^{3}}{3}\right)}\\
\\
&=&\ds{D_{t}(-u^{2})+D_{x}\left(\f{2}{3}u^{3}\right)}.
\ea
$$
Then $C=(-u^{2},\f{2}{3}u^{3})$ provides a conserved vector for (\ref{4.1.1}). This conservation law was established in \cite{igor1}.

\subsubsection{Singular second order evolution equation}

A Lie point symmetry generator of singular equation
\bb\label{5.1.10}
u_{t}=\f{u_{xx}}{u}
\ee
is given by
\bb\label{5.1.11}
X=t\f{\p}{\p t}+u\f{\p}{\p u}.
\ee

From (\ref{2.1.10})
$$
\ba{lcl}
C^{0}&=&\ds{\f{tvu_{xx}}{u}+vu},\\
\\
C^{1}&=&\ds{(u-tu_{x})\left(\f{v_{x}}{u}-\f{vu_{x}}{u^{2}}\right)-\f{v}{u}D_{x}(u-tu_{t})}.
\ea
$$
Setting $v=u$ in $C^{0}$ and $C^{1}$ and after reckoning we obtain $C=(u^{2},-2u_{x})$ as a conserved vector for equation (\ref{5.1.10}).

Considering the Lie point symmetry generator
$$Y=x\f{\p}{\p x}+2t\f{\p}{\p t}$$
the components of the conserved vector is
\bb\label{5.1.1.x}
\ba{lcl}
C^{0}&=&\ds{-\f{2tv}{u}u_{xx}-xvu_{x}},\\
\\
C^{1}&=&\ds{xvu_{t}-\f{xv}{u}u_{xx}-\f{xv_{x}u_{x}}{u}+\f{xvu_{x}^{2}}{u^{2}}-\f{2tv_{x}u_{t}}{u}+\f{2tvu_{t}u_{x}}{u^{2}}+\f{v}{u}D_{x}(xu_{x}+2tu_{t})}.
\ea
\ee
After a tedious calculus and substituting $v=u$ into (\ref{5.1.1.x}) we obtain the vector $C=(u^{2}/2,-u_{x})$. Then $Y$ does not give a new conservation law for (\ref{5.1.10}).
\subsubsection{The Korteweg--de Vries equation}

Let us now consider the Korteweg--de Vries equation
\bb\label{4.2.1}
u_{t}=u_{xxx}+uu_{x}.
\ee

It is clear that
$$
X=t\f{\p}{\p t}-\f{\p}{\p u}
$$
is a Lie point symmetry generator of (\ref{4.2.1}). Since (\ref{4.2.1}) is self-adjoint, from (\ref{2.1.10}) and setting $v=u$, we obtain
$$
\ba{lcl}
C^{0}&=&\ds{-u-tuu_{x}=-u+D_{x}\left(-t\f{u^{2}}{2}\right)},\\
\\
C^{1}&=&
\ds{tuu_{t}+u_{xx}+u^{2}=tD_{t}\left(\f{u^{2}}{2}\right)+u^{2}+u_{xx}}.
\ea
$$

Then $$D_{t}C^{0}+D_{x}C^{1}=D_{t}(-u)+D_{x}\left(\f{u^{2}}{2}+u_{xx}\right),$$
we conclude that $C=(-u,\f{u^{2}}{2}+u_{xx})$ is a conserved vector for (\ref{4.2.1}). This example was presented in the seminal work \cite{ib2}.

\subsubsection{Generalized Korteweg--de Vries equation}

{\bf Example 3}: Let us consider the following generalization of the Korteweg--de Vries equation (\ref{4.2.1})
\bb\label{4.2.2}
u_{t}=u_{xxx}+u^{\mu}u_{x},
\ee
where $\mu\neq 0$ is a constant. 

Here we shall present in a more detailed form the conservation law for (\ref{4.2.2}) arising from the Lie point symmetry generator
$$
X_{\mu}=\f{2}{\mu}u\f{\p}{\p u}-3t\f{\p}{\p t}-x\f{\p}{\p x}.
$$

From (\ref{2.1.10}) is obtained
\bb\label{4.2.4}
\ba{lcl}
A^{0}&=&\ds{v \left(3tu_{xxx}+3tu^{\mu}u_{x}+xu_{x}\f{2}{\mu}u\right)},\\
\\
A^{1}&=&\ds{-v\left(\f{2}{\mu}u^{\mu+1}+xu_{t}+3tu^{\mu}u_{t}+2\f{\mu+1}{\mu}u_{xx}+3tu_{txx}\right)+v_{x}\left(\f{2+\mu}{\mu}u_{x}+3tu_{tx}+xu_{xx}\right)}\\
\\
&&\ds{-v_{xx}\left(\f{2}{\mu}u+3tu_{t}+xu_{x}\right)}.
\ea
\ee

Setting $v=u$ in (\ref{4.2.4}) and after reckoning, we obtain
$$D_{t}A^{0}+D_{x}A^{1}=D_{t}\left(\f{4-\mu}{2\mu}u^{2}\right)+D_{x}\left[\f{\mu-4}{\mu(\mu+2)}u^{\mu+2}+\f{\mu-4}{\mu}uu_{xx}-\f{\mu-4}{2\mu}u_{x}^{2}\right].$$

Then $C=(C^{0},C^{1})$ provides a conserved vector for the generalized Korteweg--de Vries equation (\ref{4.2.2}), where 
$$
C^{0}=u^{2},\,\,\,\,C^{1}=u_{x}^{2}-2uu_{xx}-\f{2}{\mu+2}u^{\mu+2}.
$$

In particular, whenever $\mu=1$, $C=(u^{2},u_{x}^{2}-2uu_{xx}-\f{2}{3}u^{3})$ is another conserved vector for the KdV equation (\ref{4.2.1}), see \cite{ib2}.

Choosing $\mu=2$, then $C=(u^{2},u_{x}^{2}-2uu_{xx}-\f{1}{2}u^{4})$ is a conserved vector for the modified KdV equation
$$u_{t}=u_{xxx}+u^{2}u_{x}.$$

\section{Self-adjoint adjoint equations of fourth-order}\label{comment}


Concerning equation (\ref{1.1.3}), in \cite{ib1} is proved that equation $(\ref{1.1.3})$ is self-adjoint if and only if 
\bb\label{5.1.1}
g=h+\f{1}{u}(uf)',\,\,\,\,\,\,d=\f{c_{1}}{u}+\f{1}{u}(uh)'
\ee
and
\bb\label{5.1.2}
q=\f{1}{u}[c_{2}+(up)'],
\ee
where $f,\,h$ and $p$ are arbitrary functions of $u$ (see \cite{ib1}, Theorem 3.2, p.p 310).

From Theorem \ref{teo1} we conclude that equations (\ref{5.1.2}) and (\ref{3.1.2}) cannot be compatible whenever $c_{2}\neq 0$. In fact, the correct statement is
\begin{theorem}\label{teo5.2}
Equation $(\ref{1.1.3})$ is self-adjoint if and only if $g$ and $d$ are given by $(\ref{5.1.1})$
and $q$ is given by $(\ref{3.1.2})$, where $f,\,h$ and $p$ are arbitrary functions of $u$ and $c_{1}$ is an arbitrary constant.
\end{theorem}

\begin{proof}
From the self-adjointness condition (\ref{ast}) we obtain the following system of equations
\bb\label{5.1.3}
(uf)'-ug+uh=0,
\ee
\bb\label{5.1.4}
(uf)''-(ug)'+(uh)'=0,
\ee
\bb\label{5.1.5}
3(uf)'''-3(ug)''+(uh)''+2(ud)'=0,
\ee
\bb\label{5.1.6}
(up)'-uq=0,
\ee
\bb\label{5.1.7}
(up)''-(uq)'=0
\ee
and
\bb\label{5.1.8}
(uf)''''-(ug)'''+(ud)''=0.
\ee

From (\ref{5.1.3}) and (\ref{5.1.5}) we obtain (\ref{5.1.1}). Equation (\ref{5.1.4}) is a consequence of (\ref{5.1.3}). Equation (\ref{5.1.8}) is a consequence of (\ref{5.1.3}) and (\ref{5.1.5}). 

From (\ref{5.1.7}) we obtain (\ref{5.1.2}). However, substituting (\ref{5.1.2}) into (\ref{5.1.6}) we conclude that $c_{2}=0$. Thus we obtain (\ref{3.1.2}).
\end{proof}

From theorems \ref{teo5.2} and \ref{teo1}, we have the following generalization for Theorem \ref{teo5.2} (and Theorem 3.2 of \cite{ib1}):
\begin{theorem}\label{teo5.3}
Equation
\bb\label{5.1.9}
\ba{l}
u_{t}+f(u)u_{xxxx}+g(u)u_{x}u_{xxx}-r(u)u_{xxx}+h(u)u_{xx}^{2}\\
\\
+d(u)u_{x}^{2}u_{xx}-p(u)u_{xx}-q(u)u_{x}^{2}-a(u)u_{x}+b(u)=0
\ea
\ee
is self-adjoint if and only if $g$ and $d$ are given by $(\ref{5.1.1})$, $r,\,q$ and $b$ are given by $(\ref{3.1.1'})$, $(\ref{3.1.2})$ and $(\ref{3.1.3})$, respectively, where $a_{1},\,a_{2},\,a_{3}$ and $c_{1}$ are arbitrary constants and $f,\,h$ and $p$ are arbitrary functions of $u$.
\end{theorem}

\section{Conclusion}

In this paper the self-adjoint subclasses of equation (\ref{1.1.1}) was obtained. Thanks to the recently proposed  conservation theorem from Ibragimov, some conservation laws of particular self-adjoint equations type (\ref{3.2.1}) were established. Further examples can be found in \cite{ib1, ib2,igor1, ya}.

A comment in a recently published result (see \cite{ib1}, Theorem 3.2) was given in section \ref{comment}. In particular the self-adjointness condition obtained by Bruzón, Gandarias and Ibragimov to equation (\ref{1.1.3}) was generalized to equation (\ref{5.1.9}). Equation (\ref{5.1.9}) covers a wider list of equations, for instance, all equations mentioned in the present paper, the thin film equation and so on, see \cite{qu, ib1}.

The main results are summarized by Theorem \ref{teo1} and Theorem \ref{teo5.3}. In particular, by using Theorem \ref{teo5.3} and the new conservation theorem presented in \cite{ib2}, conservation laws for lubrication equation, Korteweg--de Vries and inviscid Burgers equation, among others, can be established, as pointed out in \cite{ib1, igor1, ya}.

\end{document}